\documentclass[11pt,english,oneside]{amsart}
\usepackage[T1]{fontenc}
\usepackage[latin9]{inputenc}
\usepackage{amsthm}
\usepackage{amstext}
\usepackage{amssymb}

\makeatletter
\numberwithin{equation}{section}
\numberwithin{figure}{section}
\theoremstyle{plain}
\newtheorem{thm}{Theorem}
  \theoremstyle{plain}
  \newtheorem{conjecture}[thm]{Conjecture}
  \theoremstyle{plain}
  \newtheorem{prop}[thm]{Proposition}
  \theoremstyle{plain}
  \newtheorem{lem}[thm]{Lemma}
  \theoremstyle{plain}
  \newtheorem{cor}[thm]{Corollary}
 \newtheorem{rem}[thm]{Remark}

\usepackage{amsfonts}

\usepackage{amscd}

\usepackage{graphics}

\usepackage{hyperref}

\hypersetup{colorlinks,citecolor=blue}
\input xy
\xyoption{all}

\newcommand{\BR}{{\mathbb{R}}}

\newcommand{\gC}{\Gamma}

\newcommand{\gS}{\Sigma}

\newcommand{\ti}[1]{\widetilde{#1}}

\newcommand{\vol}{\text{vol}}

\newcommand{\GL}{\text{GL}}

\numberwithin{equation}{section} 

\makeatother

\usepackage{babel}

\begin{document}

\title{Milnor-Wood inequalities for products}

\author{Michelle Bucher and Tsachik Gelander}

\begin{abstract}

We prove Milnor-wood inequalities for local products of manifolds. As a consequence, we establish the generalized Chern Conjecture for products $M\times \Sigma^k$ for any product of a manifold $M$ with a product of $k$ copies of a surface $\Sigma$ for $k$ sufficiently large. 
\end{abstract}

\thanks{Michelle Bucher acknowledges support from the Swiss National Science Foundation grant PP00P2-128309/1. Tseachik Gelander acknowledges support of the Israel Science Foundation and the European Research Council.}

\maketitle
\section{Introduction}

Let $M$ be an $n$-dimensional topological manifold. Consider the Euler class $\varepsilon_n(\xi)\in H^n(M,\mathbb{R})$ and Euler number $\chi(\xi)=\langle \varepsilon_n(\xi),[M]\rangle $ of  oriented $\mathbb{R}^n$-vector bundles over $M$. We say that the manifold $M$ satisfies a Milnor-Wood inequality with constant $c$ if for every {\it flat}  oriented $\mathbb{R}^n$-vector bundles $\xi$ over $M$, the inequality
$$ |\chi(\xi)| \leq c\cdot |\chi(M)|$$
holds. Recall that a bundle is flat if it is induced by a representation of the fundamental group $\pi_1(M)$. We denote by $MW(M)\in \mathbb{R} \cup \{+\infty \}$ the smallest such constant. 

If $X$ is a simply connected Riemannian manifold, we denote by $\widetilde{MW}(X)\in \mathbb{R} \cup \{+\infty \}$ the supremum of the values of $MW(M)$ when $M$ runs over all closed quotients of $X$. 

Milnor's seminal inequality \cite{Mi58} amounts to showing that the Milnor-Wood constant of the hyperbolic plane $\mathcal{H}$ is $\widetilde{MW}(\mathcal{H})=1/2$, and in \cite{BuGe09}, we showed that $\widetilde{MW}(\mathcal{H}^n)=1/2^n$.

We prove a product formula for the Milnor-Wood inequality valid for any closed manifolds:

\begin{thm}\label{thm:MWProd}
For any pair of compact manifolds $M_1,M_2$
$$
 \text{MW}(M_{1}\times M_{2}) = \text{MW}(M_{1})\cdot \text{MW}(M_{2}).
$$
\end{thm}

For the product formula for universal Milnor-Wood constant, we restrict to Hadamard manifolds: 

\begin{thm}\label{thm Hadamard}
Let $X_1,X_2$ be Hadamard manifolds. Then
$$
 \widetilde{\text{MW}}(X_1\times X_2)= \widetilde{\text{MW}}(X_1) \widetilde{\text{MW}}(X_2). 
$$
\end{thm}



One important application of Milnor-Wood inequalities is to make progress on the generalized Chern Conjecture. 

\begin{conjecture}[Generalized Chern Conjecture] Let $M$ be a closed oriented aspherical manifold. If the tangent bundle $TM$ of $M$ admits a flat structure then $\chi(M)=0$.
\end{conjecture}

As the name indicates, this conjecture implies the classical Chern conjecture for affine manifolds predicting the vanishing of the Euler characteristic of affine manifolds. This is because an affine structure on $M$ induces a flat structure on the tangent bundle $TM$. 

As pointed out in \cite{Mi58}, if $MW(M)<1$ then the Generalized Chern Conjecture holds for $M$. Indeed, if $\chi(M)\neq 0$ the inequality
$$ |\chi(M)|=|\chi(TM)|\leq MW(M) \cdot |\chi(M)| < |\chi(M)|$$
leads to a contradiction. 

One can use Theorem \ref{thm:MWProd} to extend the family of manifolds satisfying the Generalized Chern Conjecture. For instance: 

\begin{cor}\label{cor1}
Let $M$ be a manifold with $MW(M)<+\infty$. Then the product $M\times\Sigma^{k}$,
where $\Sigma$ is a surface of genus $\geq2$ and $k>\mbox{log}_{2}(MW(X))$ satisfies the Generalized Chern Conjecture. In particular, if $\chi(M)\neq 0$, then  $M\times\Sigma^{k}$
does not admit an affine structure.
\end{cor}

\begin{rem}
1. One can replace $\gS^k$ in Corollary \ref{cor1} by any $\mathcal{H}^k$-manifold.

2. The corollary is somehow dual to a question of Yves Benoist \cite [Section 3, p. 19]{Be00} asking wether for every closed manifold $M$ there exists $m$ such that $M\times S^m$ admits an affine structure. For example, for any hyperbolic manifold $M$, the product $M\times S^1$ admits an affine structure, but in general $m=1$ is not enough. Indeed, $\mbox{Sp(2,1)}$ has
no nontrivial $9$-dimensional representations and the dimension of
the associated symmetric space is $8$.

\end{rem}

Note that since there are only finitely many isomorphism classes of oriented $\mathbb{R}^n$-bundles which admit a flat structure, it is immediate that the set
$$ \{ |\chi(\xi)| \mid \xi \mathrm{\ is \ a \ flat \ oriented \ }  \mathbb{R}^n \mathrm{-bundle \ over \ }M\}$$
is finite for every $M$. In particular, if $\chi(M)\neq 0$, there exists a finite Milnor-Wood constant $MW(M)<+\infty$. 

However, in general, the Milnor-Wood constant can be infinite. Indeed, the implication $\chi(M)=0\Rightarrow\chi(\xi)=0$, for $\xi$ a flat oriented $\mathbb{R}^n $-bundle, does not hold in general. In Section \ref{sec:example} we exhibit a flat bundle $\xi$ with $\chi(\xi)\neq 0$ over a manifold $M$ with $\chi(M)=0$. This example is inspired by  Smillie's counterexample of the Generalized Chern Conjecture \cite{Sm77} for nonaspherical manifolds, and likewise this manifold is nonaspherical. 

The following questions are quite natural: 
\begin{enumerate}
\item Does there exist a finite constant $c(n)$ depending on $n$ only such that $MW(M)\leq c(n)$ for every closed aspherical $n$-manifold?
\item Let $X$ be a contractible Riemannian manifold such that there exists a closed $X$-manifold $M$ with $MW(M)<\infty$. Is $\ti{MW}(X)$ necessarily finite?
\item Does $\chi(M)=0\Rightarrow\chi(\xi)=0$ for flat t oriented $\mathbb{R}^n $-bundles $\xi$ over aspherical manifolds $M$?
\end{enumerate}


\section{Representations of products}

\begin{lem}\label{lem:prod-rep}
Let $H_1,H_2$ be groups and $\rho:H_1\times H_2\to\GL_n(\BR)$ a representation of the direct product and suppose that $\rho(H_i)$ is non-amenable for both $i=1,2$. Then, up to replacing the $H_i$'s by finite index subgroups, either 
\begin{itemize}
\item $V=\BR^n$ decomposes as an invariant direct sum $V=V'\oplus V''$ where the restriction $\rho |{V'}=\rho_1'\otimes\rho_2'$ is a nontrivial tensor representation, or 
\item $V=V_1\oplus V_2$ where $G_i$ is scalar on $V_{i}$.\end{itemize}
\end{lem}

\begin{proof}
This can be easily deduced from the proof of  \cite[Proposition 6.1]{BuGe09}.
\end{proof}

\begin{prop}
\label{prop:prod} 
Let $H=\prod_{i=1}^{k}H_{i}$
be a direct product
of groups and let $\rho:H\to\text{GL}_{n}^{+}({\mathbb{R}})$ be an
orientable representation, where $n=\sum_{i=1}^{k}m_{i}$. Suppose that $\rho(H_i)$ is nonamenable for every $i$. Then, up
to replacing the $H_i$'s by finite index subgroups $H^{\prime}=\prod_{i=1}^{k}H_{i}^{\prime}$,
either
\begin{enumerate} 
\item there exists $1\leq i_{0}< k$ such that $V=\BR^n$ decomposes non-trivially to an invariant direct sum $V=V'\oplus V''$ and the restricted representation $\rho\mid_{(H_{i_{0}}^{\prime}\times \prod_{i> i_0}H_{i}^{\prime},V')}$
\[
 H_{i_{0}}^{\prime}\times \prod_{i> i_0}H_{i}^{\prime}\longrightarrow \GL(V')
\]
is a nontrivial tensor, or
\item the representation $\rho^{\prime}$ factors through
\[
 \rho^{\prime}:\prod_{i=1}^{k}H_{i}^{\prime}\longrightarrow \left( \prod_{i=1}^{k}\mbox{\GL}_{m'_{i}}(\BR)\right)^+\longrightarrow\GL^+_{n}(\BR),
\]
where the latter homomorphism is, up to conjugation, the canonical diagonal embedding,
and $\rho^{\prime}(H_{i}^{\prime})$ restricts to a scalar representation
on each $\GL_{m_{j}}(\BR)$, for $i\neq j$.
\end{enumerate}
Moreover, if all $m_{i}$ are even then either $m_i'<m_i$ for some $i$ or one can replace $\GL$ with
$\GL^{+}$ everywhere. 
\end{prop}

The notation $ \left( \prod_{i=1}^{k}\mbox{\GL}_{m'_{i}}(\BR)\right)^+$ stands for the intersection of $ \prod_{i=1}^{k}\mbox{\GL}_{m'_{i}}(\BR)$ with the positive determinant matrices. 

\begin{proof}
We argue by induction on $k$.
For $k=2$ the alternative is immediate from Lemma \ref{lem:prod-rep}. Suppose $k>2$. If Item $(1)$ does not hold, it follows from Lemma \ref{lem:prod-rep} that, up to replacing the $H_i$'s by some finite index subgroups, $V$ decomposes invariantly to $V=V_1\oplus V_1'$ where $\rho(H_1)$ is scalar on $V_1'$ and $\rho(\prod_{i>1}H_i)$ is scalar on $V_1$. We now apply the induction hypothesis for $\prod_{i>1}H_i$ restricted to $V_1'$.

Finally, in Case $(2)$, since $\sum m_i=n$, either $m_i'<m_i$ for some $i$ or equality holds everywhere. In the later case,
if all the $m_i$'s are even, given $g\in H_i$, since the restriction of $\rho(g)$ each $V_{j\ne i}$ is scalar, it has positive determinant. We deduce that also $\rho(g)|_{V_i}$ has positive determinant.
\end{proof}

\section{Multiplicativity of the Milnor-Wood constant for product manifolds -- A proof of Theorem \ref{thm:MWProd}} 

Let $M_1,M_2$ be two arbitrary manifolds. We prove that 
$$MW(M_1\times M_2)=MW(M_1)\cdot MW(M_2).$$
First note that the inequality $MW(M_1\times M_2)\geq MW(M_1)\cdot MW(M_2)$ is trivial. Indeed, let $\xi_1, \xi_2$ be flat oriented bundles over $M_1$ and $M_2$ respectively of the right dimension such that $|\chi(\xi_i)|=MW(M_i)\cdot |\chi(M_i)|$ for $i=1,2$. Then $\xi_1 \times \xi_2$ is a flat bundle over $M_1\times M_2$ with 
$$|\chi(\xi_1\times \xi_2)|=|\chi(\xi_1)||\chi(\xi_2)|=MW(M_1)\cdot MW(M_2)\cdot |\chi(M_1\times M_2)|.$$

For the other inequality, let $\xi$ be a flat  oriented $\mathbb{R}^{n}$-bundle over $M_1\times M_2$, where $n=\mathrm{Dim}(M_1)+\mathrm{Dim}(M_2)$. We need to show that  
\[
\left|\chi(\xi)\right|\leq MW(X_{1})\cdot MW(X_{2})\cdot \left|\chi(M)\right|.\]
Observe that if we replace $M$ by a finite cover, and the bundle
$\xi$ by its pullback to the cover, then both sides of the previous
inequality are multiplied by the degree of the covering. 

The flat bundle $\xi$ is induced by a representation \[
\rho:\pi_1(M_1\times M_2)\cong \pi_1(M_1)\times \pi_1(M_2) \longrightarrow\mbox{GL}_{n}^{+}(\mathbb{R}).\]
If $\rho(\pi_1(M_i))$ is amenable for $i=1$ or $2$, then $\rho^{*}(\varepsilon_{n})=0$ \cite[Lemma 4.3]{BuGe09} and hence $\chi(\xi)=0$ and there is nothing to prove. Thus, we can without loss of generality suppose
that, upon replacing $\Gamma$ by a finite index subgroup the representation $\rho$ factors as in Proposition \ref{prop:prod}. 

In case (1) of the proposition, we obtain that $\rho^*(\varepsilon_n)=0$ by Lemma \ref{lem:tensor} and \cite[Lemma 4.2]{BuGe09}. In case (2) we get that $\rho$ factors through
$$\rho:\pi_1(M_1)\times \pi_1(M_2)\longrightarrow \left( \mbox{\GL}_{m'_{1}}(\BR)\times \mbox{\GL}_{m'_{2}}(\BR) \right)^+\xrightarrow{\ i \ }\\GL^+_{n}(\BR),
$$
where the latter embedding $i$ is up to conjugation the canonical embedding. Furthermore, up to replacing $\rho$ by a representation in the same connected component of 
$$
 \mathrm{Rep}(\pi_1(M_1)\times \pi_1(M_2), \left( \mbox{\GL}_{m'_{1}}(\BR)\times \mbox{\GL}_{m'_{2}}(\BR)\right)^+)
$$ 
which will have no influence on the pullback of the Euler class, we can without loss of generality suppose that the scalar representations of $\pi_1(M_1)$ on $ \mbox{\GL}_{m'_{2}}$ and $\pi_1(M_2)$ on $\mbox{\GL}_{m'_{1}}$ are trivial, so that $\rho$ is a product representation. If $m'_1$ or $m'_2$ is odd, then $i^*(\varepsilon_n)=0\in H^n_c(( \mbox{\GL}_{m'_{1}}(\BR)\times \mbox{\GL}_{m'_{2}}(\BR))^+)$. If $m'_1$ and $m'_2$ are both even then Proposition \ref{prop:prod} further tells us that either $m'_i<m_i$ for $i=1$ or $2$, or the image of $\rho$ lies in $\mbox{\GL}^+_{m_{1}}(\BR)\times \mbox{\GL}^+_{m_{2}}(\BR)$. In the first case, the Euler class vanishes  \cite[Lemma 4.2]{BuGe09}, while in the second case, we immediately obtain the desired inequality. This finishes the proof of Theorem \ref{thm:MWProd}.

\section{Multiplicativity of the universal Milnor-Wood constant for Hadamard manifolds - a proof of Theorem \ref{thm Hadamard}} 
%
%
%
%
%
%
%


Theorem \ref{thm Hadamard} can be reformulated as follows:

\begin{thm}\label{thm:8} 
Let $X$ be a Hadamard manifold with de-Rham decomposition $X=\prod_{i=1}^k X_i$, then $\ti{\text{MW}}(X)=\prod_{i=1}^k\ti{\text{MW}}(X_i)$. 
\end{thm}

We shall now prove Theorem \ref{thm:8}. Note that the inequality "$\ge$" is obvious. 
Let $M=\gC\backslash X$ be a compact $X$-manifold. We must show that $\text{MW}(M)\le\prod_{i=1}^k\ti{\text{MW}}(X_i)$.
Note that $\gC$ is torsion free. Let us also assume that $k\ge 2$.
If $M$ is reducible one can argue by induction using Theorem \ref{thm:MWProd}. Thus we may assume that $M$ is irreducible.
Observe that this implies that $\text{Isom}(X)$ is not discrete.
If $\gC$ admits a nontrivial normal abelian subgroup then by the flat torus theorem (see \cite[Ch. 7]{BH}) $X$ admits an Euclidian factor which implies the vanishing of the Euler class. Assuming that this is not the case we apply the Farb--Weinberger theorem \cite[Theorem 1.3]{FaWe} to deduce that $X$ is a symmetric space of non-compact type. 
Thus, up to replacing $M$ by a finite cover (equivalently, replace $\gC$ by a finite index subgroup), we may assume that $\gC$ lies in $G=\text{Isom}(X)^\circ=\prod_{i=1}^k\text{Isom}(X_i)^\circ=\prod_{i=1}^kG_i$ and $G$ is an adjoint semisimple Lie group without compact factors and $\gC\le G$ is irreducible in the sense that its projection to each factor is dense. Denote by $\ti G_i$ the universal cover of $G_i$, and by $\ti\gC\le\prod_{i=1}^k\ti G_i$ the pullback of $\gC$. 

Let $\rho:\gC\to\GL^+_n(\BR)$ be a representation inducing a flat oriented vector bundle $\xi$ over $M$. Up to replacing $\gC$ by a finite index subgroup, we may suppose that $\rho(\gC)$ is Zariski connected. Let $S\le \GL^+_n(\BR)$ be the semisimple part of the Zariski closure of $\rho(\gC)$, and let $\rho':\gC\to S$ be the quotient representation. By superrigidity, the map $\text{Ad}\circ\rho':\gC\to \text{Ad}(S)$ extends to $\phi:\gC\le \prod_{i=1}^k G_i\to \text{Ad}(S)$ (see \cite{Margulis,Monod,GKM}). This map can be pulled to $\ti\phi:\ti\gC\to S$. Recall also that $\prod_{i=1}^k \ti G_i$ is a central discrete extension of $\prod_{i=1}^k G_i$ and, likewise, $\ti\gC$ is a central extension of $\gC$. If $n_i=\dim X_i$ and $n=\sum_{i=1}^k n_i$ we deduce from Proposition \ref{prop:prod} and Lemma \ref{lem:tensor} that either the Euler class vanishes or the image of $\ti\phi$ lies (up to decomposing the vector space $\BR^n$ properly) in $(\prod_{i=1}^k\GL_{n_i})^+$. 

Suppose that $\text{MW}(X_i)$ is finite for all $i=1,\ldots,k$ and let $M_i$ be closed $X_i$-manifolds. Let $\xi'$ be the flat vector bundle on $\prod_{i=1}^kM_i$ coming from $\ti\rho$ reduced to $\prod_{i=1}^kM_i$, and let $\xi'_i$ be the vector bundle on $M_i$ induced by $\ti\rho_i,~i=1,\ldots,k$. By Lemma \ref{lem:14}, we have
$$
 \frac{\chi(\xi)}{\vol(M)}=\frac{\chi(\xi')}{\vol(\prod_{i=1}^kM_i)}=\prod_{i=1}^k\frac{\chi(\xi_i')}{\vol(M_i)}\le
 \prod_{i=1}^k\text{MW}  (X_i),
$$ 
which finishes the proof of Theorem \ref{thm:8}. \qed


\section{Example: a flat bundle with nonzero Euler number over a manifold with zero Euler characteristic}\label{sec:example}

Recall that given two closed manifolds of even dimension, the Euler
characteristic of connected sums behaves as\[
\chi(M_{1}\sharp M_{2})=\chi(M_{1})+\chi(M_{2})-2.\]
The idea is to find $M=M_{1}\sharp M_{2}$ such that $M_{1}$ admits
a flat bundle with nontrivial Euler number in turn inducing such a
bundle on the connected sum, and to choose then $M_{2}$ in such a
way that the Euler characteristic of the connected sum vanishes. Take
thus\[
M_{1}=\Sigma_{2}\times\Sigma_{2},\; M_{2}=(S^{1}\times S^{3})\sharp(S^{1}\times S^{3})\;\mbox{and}\; M=M_{1}\sharp M_{2}.\]
These manifolds have the following Euler characteristics: \begin{eqnarray*}
\chi(M_{1}) & = & 4,\\
\chi(M_{2}) & = & 2\chi(S^{1}\times S^{3})-2=-2,\\
\chi(M) & = & 0.\end{eqnarray*}
Let $\eta$ be a flat bundle over $\Sigma_{2}$ with Euler number
$\chi(\eta)=1$. (Note that we know that such a bundle exists by \cite{Mi58}.)
Let $f:M\rightarrow M_{1}$ be a degree $1$ map obtained by sending
$M_{2}$ to a point, and consider \[
\xi=f^{*}(\eta\times\eta).\]
Obviously, since $\eta$ is flat, so is the product $\eta\times\eta$
and its pullback by $f$. Moreover, the Euler number of $\xi$ is
\[
\chi(\xi)=\chi(\eta\times\eta)=1.\]
Indeed, the Euler number of $\eta\times\eta$ is the index of a generic
section of the bundle, which we can choose to be nonzero on $f(M_{2})$,
so that we can pull it back to a generic section of $\xi$ which will
clearly have the same index as the initial section on $\eta\times\eta$.

\section{Proportionality principles and vanishing of the Euler class of tensor products}



\begin{lem}\label{lem:14}
Let $X$ be a simply connected Riemannian manifold, $G= \mbox{Isom}(M)$ and $\rho:G\to\GL_n^+(\BR)$ a representation.
Then $\frac{\chi (\xi_\rho )}{\vol(M)}$, where $M=\gC\backslash X$ is a closed $X$-manifold and $\xi_\rho$ is the flat vector bundle induced on $M$ by $\rho$ restricted to $\gC$, is a constant independent of $M$.
\end{lem}

\begin{proof}
There is a canonical isomorphism $H^*_c(G)\cong H^*(\Omega^*(X)^G))$ between the continuous cohomology of $G$ and the cohomology of the cocomplex of $G$-invariant differential forms $\Omega^*(X)^G$ on $X$ equipped with its standard differential. (For $G$ a semisimple Lie group, every $G$-invariant form is closed, hence one further has $H^*(\Omega^*(X)^G)\cong \Omega^*(X)^G$.) In particular, in top dimension $n=\mathrm{dim}(X)$, the cohomology groups are $1$-dimensional $H^n_c(G)\cong H^n(\Omega^*(X)^G))\cong \mathbb{R}$ and contain the cohomology class given by the volume form $\omega_X$.  

Since the bundle $\xi_\rho$ over $M$ is induced by $\rho$, its Euler class $\varepsilon_n(\xi_\rho)$ is the image of $\varepsilon_n\in H^n_c(\mathrm{GL}^+(\mathbb{R},n)$ under
$$ H^n_c(\mathrm{GL}^+(\mathbb{R},n)\xrightarrow{\  \rho^*  }  H^n_c(G) \longrightarrow H^n(\Gamma)\cong H^n(M), $$
where the middle map is induced by the inclusion $\Gamma \hookrightarrow G$. In particular, $\rho^*(\varepsilon_n)=\lambda\cdot [\omega_X]\in H^n_c(G)$ for some $\lambda \in \mathbb{R}$ independent of $M$. It follows that $\chi(\xi_\rho)/\mathrm{Vol}(M)=\lambda$. 
\end{proof}

\begin{lem}\label{lem:tensor} Let $\rho_\otimes:\GL^+(n,\BR)\times \GL^+(m,\BR)\rightarrow \GL^+(nm,\BR)$ denote the tensor representation. If $n,m\geq 2$, then 
$$\rho_\otimes^*(\varepsilon_{nm})=0\in H_c^{nm}(\GL(n,\BR)\times \GL(m,\BR)).$$
\end{lem}

\begin{proof} The case $n=m=2$ was proven in \cite[Lemma 4.1]{BuGe09}, based on the simple observation that interchanging the two $\GL^+(2,\BR)$ factors does not change the sign of the top dimensional cohomology class in $H_c^4(\GL(2,\BR)\times \GL(2,\BR))\cong \BR$, but it changes the orientation on the tensor product, and hence the sign of the Euler class in $H^4_c(\GL^+(4,\BR))$. 

Let us now suppose that at least one of $n,m$ is strictly greater than $2$, or equivalently, that $n+m<nm$. The Euler class is in the image of the natural map 
$$H^{nm}(B\GL(nm,\BR))\longrightarrow H^{nm}_c(\GL(nm,\BR)).$$
By naturality, we have a commutative diagram
 \[
\xymatrix{H^{nm}(B\GL^+(nm,\BR)) \ar^{\rho^*_{\otimes}}[d]\ar^{}[r] &H^{nm}_c(\GL^+(nm,\BR)) \ar^{\rho^*_{\otimes}}[d]\\
H^{nm}(B(\GL^+(n,\BR)\times \GL^+(m,\BR)))\ar^{}[r] &H^{nm}_c(\GL^+(n,\BR)\times \GL^+(m,\BR))). }
\]
Since the image of the lower horizontal arrow is contained in degree $\leq n+m$, it follows that $\rho_\otimes^*(\varepsilon_{nm})=0$.
\end{proof}


\end{document}